\documentclass[11pt,reqno]{amsart}
\usepackage{amsmath, amssymb, amsthm}
\usepackage{url}
\usepackage[breaklinks]{hyperref}
\usepackage{BOONDOX-uprscr}

\renewcommand{\(}{\left\(}
\renewcommand{\)}{\right\)}
\renewcommand{\[}{\left\[}
\renewcommand{\]}{\right\]}

\numberwithin{equation}{section}
\usepackage{mathtools}

\theoremstyle{plain}
\newtheorem{theorem}{Theorem}[section]
\newtheorem{lemma}[theorem]{Lemma}

\newtheorem{corollary}[theorem]{Corollary}

   \makeatletter
\def\proof{\@ifnextchar[{\@oproof}{\@nproof}}
\def\@oproof[#1][#2]{\trivlist\item[\hskip\labelsep\textit{#2 Proof of\
#1.}~]\ignorespaces}
\def\@nproof{\trivlist\item[\hskip\labelsep\textit{Proof.}~]\ignorespaces}

\makeatother

\begin{document}
\title{A divisor generating $q$-series and cumulants arising from random graphs}

\thanks{$2020$ \textit{Mathematics Subject Classification.} Primary 11P84,  33D15; Secondary 05C80,  60F99.}

\author{Archit Agarwal}
\address{Archit Agarwal, Department of Mathematics, Indian Institute of Technology Indore, Simrol, Indore 453552, Madhya Pradesh, India.}
\email{phd2001241002@iiti.ac.in, archit.agrw@gmail.com}

\author{Subhash Chand Bhoria}
\address{Subhash Chand Bhoria, Pt. Chiranji Lal Sharma Government PG College, Karnal, Urban Estate, Sector-14, Haryana 132001, India.}
\email{scbhoria89@gmail.com}

\author{Pramod Eyyunni}
\address{Pramod Eyyunni,  Department of Mathematics,  Birla Institute of Technology And Science Pilani,  Pilani Campus,  Pilani-333031, Rajasthan, India.}
\email{pramod.eyyunni@pilani.bits-pilani.ac.in,  pramodeyy@gmail.com}

\author{Bibekananda Maji}
\address{Bibekananda Maji, Department of Mathematics, Indian Institute of Technology Indore, Simrol, Indore 453552, Madhya Pradesh, India.}
\email{bmaji@iiti.ac.in, bibek10iitb@gmail.com}

\author{Tanay Wakhare}
\address{Tanay Wakhare,   Department of Electrical Engineering and Computer Science, MIT, Cambridge, MA 02139, United States of America.}
\email{twakhare@mit.edu}

\maketitle

\begin{center}
{Dedicated to George Andrews and Bruce Berndt for their 85th birthdays}
\end{center}

\begin{abstract}

Uchimura, in 1987,  introduced a probability generating function for a random variable $X$ and using properties of this function he discovered an interesting $q$-series identity. 
 He further showed that the $m$-th cumulant  with respect to the random variable $X$ is nothing but  the  generating function for the generalized divisor function $\sigma_{m-1}(n)$.
 Simon, Crippa, and Collenberg, in 1993, explored the $G_{n,p}$-model of a random acyclic digraph and defined a random variable $\gamma_n^{*}(1)$.  Quite interestingly,  they  found links between limit of its mean and the generating function for the divisor function $d(n)$.    Later in 1997,  Andrews, Crippa and Simon extended these results using $q$-series techniques. They calculated limit of the mean and variance of the random variable $\gamma_n^{*}(1)$  which correspond to the first and second cumulants. 
 In this paper,  we generalize the result of Andrews,  Crippa and Simon by calculating limit of the $t$-th cumulant in terms of the generalized divisor function.  Furthermore,  we also discover limit forms for identities of Uchimura and Dilcher.  This provides a fourth side to the Uchimura-Ramanujan-divisor type three way partition identities expounded by the first four authors recently.

\end{abstract}

\section{Introduction}
Let $d(n)$ and $\sigma_{m}(n)$ be the well-known divisor functions and their generating functions are given by,  $$ \sum_{n=1}^\infty d(n) q^n= \sum_{n=1}^\infty\frac{q^n}{1-q^n},  \quad  \sum_{n=1}^\infty \sigma_{m}(n) q^n = \sum_{n=1}^\infty \frac{n^m q^n}{1-q^n},   ~~\textrm{for}~~m \in \mathbb{Z}, ~ |q|<1.$$
In 1919,  Kluyver \cite{kluyver} showed that 
\begin{equation}\label{Kluyver}
\sum_{n=1}^{\infty} \frac{(-1)^{n-1} q^{\frac{n(n+1)}{2}}}{(1-q^n) ( q)_n  } = \sum_{n=1}^{\infty} \frac{ q^n }{1-q^n}.
\end{equation}
A one variable generalization of the identity \eqref{Kluyver} can be found in the second notebook of Ramanujan \cite[p.~354]{ramanujanoriginalnotebook2}.  
Uchimura \cite{uchimura81}, in 1981,  gave a new expression for Kluyver's identity.  Mainly,  he proved that
\begin{align}\label{uchimura's identity}
\sum_{n=1}^\infty nq^n(q^{n+1})_\infty=\sum_{n=1}^\infty \frac{(-1)^{n-1} q^{ \frac{n(n+1)}{2}}}{(1-q^n)(q)_n}=\sum_{n=1}^\infty\frac{q^n}{1-q^n}.
\end{align}
A beautiful partition theoretic interpretation of this identity was obtained by Bressoud and Subbarao \cite{bresub},  which has been further explored by the first four authors in \cite{ABEM}.  Over the period of time,  Uchimura's identity \eqref{uchimura's identity} has been extended by many mathematicians.  Among them,  Uchimura himself in \cite{uchimura87},     Dilcher \cite{dilcher},   Andrews-Crippa-Simon \cite{andrewssiam1997},  and  Gupta-Kumar \cite{GK21} worked on this identity.    Recently,   the first four authors \cite{ABEM24} studied these generalizations and presented a unified theory.

Now,  we shall discuss the generalizations given by Uchimura \cite{uchimura87} and Andrews-Crippa-Simon \cite{andrewssiam1997} and their applications in probability theory and random graphs.    

Uchimura \cite[p.~76]{uchimura87} considered a random variable $X$ with the probability generating function 
\begin{align*}
G(x,q)=\sum_{n=0}^\infty x^n~ \mathrm{Pr}(X=n),
\end{align*}
where $ \mathrm{Pr}(X=n)=q^n(q^{n+1})_\infty$ for any non-negative integer $n$ and $q \in (0,1)$. 
Moreover, he proved that,   for any natural number $m$, the $m$-th moment is
\begin{align}\label{Uchimura's gen}
\mathbb{E}(X^m)= M_m=Y_m(K_1,K_2,\dots,K_m),  
\end{align}
where
\begin{align}\label{define M_m and K_m}
M_m := \sum_{n=1}^{\infty} n^m q^n (q^{n+1})_{\infty}, \quad \text{and} \quad K_{m} := \sum_{n=1}^{\infty} \sigma_{m-1}(n) q^n, 
\end{align} 
and $Y_m$ is the Bell polynomial defined by
\begin{equation}\label{define Bell polynomial}
Y_m \left( u_1, u_2, \dots, u_m\right) = \sum_{\Pi(m)} \frac{m!}{k_1 ! \dots k_m !} \left( \frac{u_1}{1!} \right)^{k_1} \dots \left( \frac{u_m}{m!} \right)^{k_m},
\end{equation}
where $\Pi(m)$ denotes a partition of $m$ with
\begin{equation*}
k_1 + 2 k_2 + \cdots + mk_m = m.  
\end{equation*}
He further showed that the $m$-th cumulant $h_m$ is equal to the divisor generating function $K_m$,  that is,  
\begin{align}\label{mth cumulant}
h_m=K_m,  \quad \forall~m \in \mathbb{N},  
\end{align}
 by writing the probability generating function in the following way
\begin{align*}
G(e^t,q)=\exp\left( \sum_{m=1}^\infty h_m \frac{t^m}{m!} \right).  
\end{align*}
The above identity \eqref{mth cumulant} proves that the divisor generating function $K_m$ is nothing but the $m$-th cumulant with respect to the random variable $X$.  As we know that the second cumulant $h_2$ is same as $\mathrm{Var}(X)$,  so we have
\begin{align}\label{variance by uchimura}
\mathrm{Var}(X)=\sum_{n=1}^{\infty} \sigma(n) q^n.
\end{align}
Uchimura \cite{uchimura87} provided a combinatorial interpretation of the probability generating function $G(x,q)$ by highlighting its importance through applications in the analysis of data structures, called heaps. 
His work is primarily focused on examining the average number of exchanges needed to insert an element into a heap, offering insights into the efficiency of this operation. For readers interested in a detailed exposition, further discussion can be traced in \cite[Section~3]{uchimura87}.

In 1993, Simon, Crippa, and Collenberg \cite{SCC1993} analyzed how the transitive closure is distributed in the $G_{n,p} $--model of a random acyclic digraph. They defined the $G_{n,p}$--model as a random acyclic digraph with a vertex set $ V = \{1, 2, \dots, n\} $ and directed edges appearing between vertices $ (i, j) $, for $ 1 \leq i \leq j \leq n $, with probability $ p \in (0, 1) $. Let $ \gamma^*_n(1) $ be a random variable representing the number of vertices reachable from vertex $1$ by a directed path, including vertex $1$ itself. They initially established a probability function, for $1 \leq h \leq n$,
\begin{align}\label{probability function define by Simon, Crippa, and Collenberg}
\mathrm{Pr}(\gamma^*_n(1)=h)=q^{n-h}\prod_{j=1}^{h-1}(1-q^{n-j}),
\end{align}
where $q=1-p$. In the same paper, by treating the random variable representing the size of a node's transitive closure as a discrete-time pure-birth process, they derived an expression for its distribution, mean and variance, linking these to the divisor generating function. They proved that,   
\begin{align}\label{expectation formula given by Simon, Crippa, and Collenberg}
\lim_{n\rightarrow \infty}(n-\mathbb{E}(\gamma^*_n(1))=\sum_{n=1}^\infty d(n) q^n.  
\end{align}
 Andrews, Crippa and Simon \cite{andrewssiam1997} further studied the same random variable and   proved that
\begin{align}\label{variance by andrews crippa simon}
\lim_{n\rightarrow \infty} \mathrm{Var}(\gamma^*_n(1))=\sum_{n=1}^\infty \sigma(n)q^n.  
\end{align}
In view of \eqref{variance by uchimura} and \eqref{variance by andrews crippa simon},  we can clearly see that 
\begin{align*}
\mathrm{Var}(X)=\lim_{n\rightarrow \infty} \mathrm{Var}(\gamma^*_n(1)) = \sum_{n=1}^\infty \sigma(n)q^n =K_2,  
\end{align*}
where $X$ is the random variable studied by Uchimura.   As mentioned earlier,  Uchimura also showed that  the $m$th cumulant $h_m$ with respect to the random variable $X$ is exactly equal to the generating function  for the  generalized divisor function $\sigma_{m-1}(n)$. 

This observation led the first four authors \cite[p.~31,  Problem 2]{ABEM24} to raise the question on the interpretation of the generalized divisor generating function $K_m$,  defined in \eqref{define M_m and K_m},   with respect to the random variable $\gamma^*_n(1)$ studied by Simon, Crippa and Collenberg.   In this paper, we  completely settle this question (see Theorem \ref{Conjecture}).  


It is worthwhile to indicate that,  Andrews, Crippa and Simon \cite{andrewssiam1997} established a more general form of \eqref{expectation formula given by Simon, Crippa, and Collenberg} using the theory of $q$-series.   For that, they first defined a sequence $\{a_n(q)\}$ of polynomials in $q$, which obeys the following recurrence relation, for $n\geq 1$, 
\begin{align}\label{sequence define by andrews crippa simon}
a_n(q)=f(n)+(1-q^{n-1})a_{n-1}(q),\hspace{5mm}a_0(q)=0,
\end{align}
where $f(n)=\sum_{k\geq 0}c_kn^k$ is a non-zero polynomial in $n$ with rational coefficients. Then they showed the existence of rational coefficients $h_j$,  such that
\begin{align}\label{limit formula given by andrews crippa simon}
\lim_{n\rightarrow \infty} \left( \sum_{i=1}^n f(i)-a_n(q) \right) = \sum_{j\geq 1} h_jP_j(q),
\end{align}
where $P_j(q):=P_j(K_1(q),  K_2(q),\dots,K_{j}(q))$ is some polynomial with rational coefficients,
and the coefficients $h_j$ can be evaluated by 
\begin{align*}
h_1=c_0,\hspace{5mm} h_j=\sum_{i\geq j-1}(-1)^{i-j+1} {i-1 \choose j-2} i! \sum_{k\geq i}c_k \bar{s}(k,i),
\end{align*}
where $\bar{s}(k,i)$ is the Stirling number of the second kind. Further, they observed that the expected value of $\gamma^*_n(1)$, denoted as $\mathbb{E}(\gamma^*_n(1))$, satisfies the recurrence relation given in \eqref{sequence define by andrews crippa simon} with $f(n) = 1$ for all natural numbers $n$. Using \eqref{limit formula given by andrews crippa simon}, they derived \eqref{expectation formula given by Simon, Crippa, and Collenberg} and \eqref{variance by andrews crippa simon}.

The paper is organized as follows.  In Section \ref{Main Results},  we state main results of our paper.  
Next in Section \ref{Preliminary},  we provide some preliminary results along with a crucial lemma which is used in the proof of Theorem \ref{a new q-series identity}.  This theorem leads to a new term for the identity \eqref{limit formula given by andrews crippa simon} established by Andrews, Crippa, and Simon,  see Theorem \ref{new expression to Andrews Crippa Simon}.  Section \ref{proofs} is devoted to the proof of the main results,  which include a new limit form for Uchimura's generalization \eqref{Uchimura's gen} as well as Dilcher's generalization \eqref{dilcher 1} of Uchimura's identity \eqref{uchimura's identity}.    
We prove a general result for the limit of the $t$-th cumulant in terms of the generalized divisor function in Section \ref{general case}. 
Finally,  in Section \ref{remarks},  we finish our paper by giving some concluding remarks.




\section{Main Results}\label{Main Results}
Before stating our main results,  we state a result  \cite[Corollary 2.14]{ABEM24} obtained by the first four authors recently,  which gives a Ramanujan-type expression for Uchimura's idenity \eqref{Uchimura's gen},  namely,   for any $k \in \mathbb{N}\cup \{0\}$,  we have 
\begin{align}\label{Ramanujan-type for Uchimura}
\sum_{n=1}^\infty n^k q^n(q^{n+1})_\infty =  \sum_{n=1}^\infty \frac{(-1)^{n-1} q^{ {n+1 \choose 2 } } A_k(q^n)}{(1-q^n)^k (q)_n}, 
\end{align}
where $A_k(x)$ denotes the Eulerian polynomial of degree $k$ defined by the following generating function
\begin{align*}
\sum_{k=0}^\infty A_{k}(x) \frac{t^k}{k!} = \frac{(1-x) e^t }{e^{xt} - x e^t}.  
\end{align*}
Now we are ready to state the main results of our paper.  Motivated from identity \eqref{Ramanujan-type for Uchimura},  we obtain the identity below.    
\begin{theorem} \label{a new q-series identity}
For any non-negative integer $k$, we have the following expression
\begin{align*}
\sum_{n=1}^\infty \left(\sum_{m=1}^n m^k\right) q^n(q^{n+1})_\infty = \sum_{n=1}^\infty \frac{(-1)^{n-1} q^{{n+1 \choose 2}}A_k(q^n)}{(1-q^n)^{k+1} (q)_n}.
\end{align*}
\end{theorem}
Utilizing this result,  we obtain a new expression for the identity \eqref{limit formula given by andrews crippa simon} of Andrews, Crippa and Simon.

\begin{theorem}\label{new expression to Andrews Crippa Simon}
Let $a_n(q)$ be a polynomial in $q$ defined by the recurrence relation $$a_0(q)=1, \hspace{5mm}a_n(q)=f(n)+(1-q^{n-1})a_{n-1}(q),\hspace{5mm} \mathrm{for} \hspace{5mm} n\geq 1,$$ where $f(n)=\sum_{k\geq 0} c_k n^k$ is a polynomial in $n$. Then
\begin{align*}
\lim_{n\rightarrow \infty}\left(\sum_{i=1}^n f(i)-a_n(q)\right) = \sum_{n=1}^\infty \left(\sum_{i=1}^n f(i)\right)q^n(q^{n+1})_\infty.
\end{align*}
\end{theorem}
We present two applications of this theorem.  First, we give  a limit form for Uchimura's function $M_k$,  defined in  \eqref{define M_m and K_m} and the second application presents a limit form for Dilcher's identity \eqref{dilcher 1}.  
\begin{corollary}\label{a new term of limit to our generalization}
Let $k$ be a non-negative integer and $a_{n,k}(q)$ be a sequence of polynomials in $q$ defined by the recurrence relation 
\begin{align}\label{definition a_n,k}
a_{0,k}(q)=1, \hspace{5mm}a_{n,k}(q)=f_k(n)+(1-q^{n-1})a_{n-1,k}(q),\hspace{5mm} \mathrm{for} \hspace{5mm} n\geq 1,
\end{align}
where $$ f_k(n)=\sum_{j=1}^k {k \choose j}(-1)^{j-1}n^{k-j}. $$ Then
\begin{align*}
\lim_{n\rightarrow \infty}\left(n^k- a_{n,k}(q)\right) = \sum_{n=1}^\infty n^k q^n(q^{n+1})_\infty.
\end{align*}
\end{corollary}
Now we state  an interesting generalization of Uchimura's  identity \eqref{uchimura's identity} due to Dilcher \cite[Equations (4.3),(5.7)]{dilcher}.  For for $k\in \mathbb{N}$,
\begin{align}\label{dilcher 1}
\sum_{n=k}^\infty {n\choose k} q^n(q^{n+1})_\infty=q^{-{k\choose 2}}\sum_{n=1}^\infty \frac{(-1)^{n-1}q^{{n+k \choose 2}}}{(1-q^n)^k (q)_n}=\sum_{j_1=1}^\infty\frac{q^{j_1}}{1-q^{j_1}} \cdots \sum_{j_k=1}^{j_{k-1}}\frac{q^{j_k}}{1-q^{j_k}}.
\end{align}
As an application of Theorem \ref{new expression to Andrews Crippa Simon},  we obtain a limit expression for Dilcher's identity \eqref{dilcher 1}.  
\begin{corollary}\label{Limiting exp for Dilcher}
Let $k$ be a non-negative integer and $a_{n,k}(q)$ be a sequence of polynomials in $q$ defined by the recurrence relation
\begin{align}\label{a_n,k for Dilcher}
a_{0,k}(q)=1,   a_{n,k}(q) = f_k(n) +  (1- q^{n-1}) a_{n-1,k}(q), ~~ \textrm{for}~~ n \geq 1,  
\end{align}
where 
\begin{align}\label{f_k(n) for Dilcher}
f_k(n) = {n-1 \choose k-1}. 
\end{align}
Then we have
\begin{align}\label{limit form_Dilcher}
\lim_{n \rightarrow \infty} \left(  {n \choose k} - a_{n, k}(q)  \right)  = \sum_{n=k}^\infty {n\choose k} q^n(q^{n+1})_\infty.   
\end{align}

\end{corollary}
Before proceeding to the next result,  let $\gamma_n^*(1)$ be the random variable, studied by Simon-Crippa-Collenberg, as defined just above equation \eqref{probability function define by Simon, Crippa, and Collenberg} in the introduction.     For a random variable $X$,  it is well-known that the cumulant generating function is given by 
\begin{align}\label{cumulant generating function}
\log\left( \mathbb{E}[e^{u X}] \right) = \sum_{t=1}^\infty \kappa_t(X) \frac{u^t}{t!},
\end{align}
where $\kappa_t(X)$ is the $t$-th cumulant with respect to the random variable $X$.  

Simon et al. \cite[p.~7,  Equation (18)]{SCC1993} and later Andrews et al.  \cite[p.~52,  Equation (36)]{andrewssiam1997} proved that
\begin{align}
\lim_{n\rightarrow \infty}(n- \kappa_1(\gamma_n^*(1)))&=\lim_{n\rightarrow \infty}(n-\mathbb{E}(\gamma^*_n(1))=\sum_{n=1}^\infty d(n) q^n=K_1(q),\label{kappa_1} \\
\lim_{n\rightarrow \infty}\kappa_2(\gamma_n^*(1))&=\lim_{n\rightarrow \infty} \mathrm{Var}(\gamma^*_n(1))=\sum_{n=1}^\infty \sigma(n)q^n=K_2(q). \label{kappa_2}
\end{align}
Here in this paper, we first calculate the limit of thrid, fourth and fifth cumulant and then state a general result for higher cumulants.

\begin{theorem}\label{theorem for kappa_3, kappa_4 and kappa_5}
We have
\begin{align}
\lim_{n\rightarrow \infty}\kappa_3(\gamma_n^*(1))&= -K_3(q), \label{kappa_3} \\
\lim_{n\rightarrow \infty}\kappa_4(\gamma_n^*(1))&= K_4(q), \label{kappa_4} \\
\lim_{n\rightarrow \infty}\kappa_5(\gamma_n^*(1))&= -K_5(q). \label{kappa_5}
\end{align}
\end{theorem}
More generally,  we have the following result. 
\begin{theorem}\label{Conjecture}
For any natural number $t$ with $t>1$, we have
\begin{align*}
\lim_{n\rightarrow \infty}\kappa_t(\gamma_n^*(1))&= (-1)^t K_t(q).
\end{align*}
\end{theorem}
In the next section,  we collect a few well-known results which will be useful throughout the paper.

\section{Preliminary Tools}\label{Preliminary}
The generating function for Bernoulli numbers is given by 
\begin{align*}
\sum_{n=0}^\infty B_n \frac{x^n}{n!} = \frac{x}{e^x-1},  \quad |x| < 2\pi.  
\end{align*}
The generating function for Eulerian polynomials \cite[p.~244]{comtet}  is as follows
\begin{align*}
\sum_{n=0}^\infty A_n(t) \frac{x^n}{n!} = \frac{t-1}{t-e^{x(t-1)}}.
\end{align*}
Eulerian polynomials satisfy the following recurrence relation:
\begin{align}\label{recurrence relation of Eulerian polynomial}
A_0(t)=1,  ~ A_k(t)=\sum_{j=0}^{k-1} {k \choose j} A_j(t)(t-1)^{k-1-j} ~~\textrm{ for}~~ k \geq 1.
\end{align}
One can also show that
\begin{align}\label{eulerian polynomial}
\sum_{n=1}^\infty n^k x^n = \frac{xA_k(x)}{(1-x)^{k+1}}.
\end{align}
Bernoulli showed that the sum of $k$-th powers of the first $n-1$ natural numbers can be explicitly written as    
\begin{align}\label{sum of k-th power} 
\sum_{m=1}^{n-1} m^k=\frac{1}{k+1}\sum_{j=0}^k {k+1 \choose j} B_j n^{k+1-j}.
\end{align}
Now we state and prove a lemma that will be crucial to prove Theorem \ref{a new q-series identity}.  This result gives a relation between Bernoulli numbers and Eulerian polynomials.  
\begin{lemma}\label{relation between B_n and A_n(t)}
For any non-negative integer $k$ and complex number $t$, we have
\begin{align*}
S_k(t):=\frac{1}{k+1}\sum_{j=0}^k {k+1 \choose j} B_j(1-t)^j A_{k+1-j}(t)=tA_k(t).
\end{align*}
\end{lemma}
\begin{proof}
 Let us consider the generating function for $S_k(t)$,  
{\allowdisplaybreaks
\begin{align*}
\sum_{k=0}^\infty S_k(t) \frac{x^{k+1}}{k!} &= \sum_{k=0}^\infty \frac{1}{k+1}\sum_{j=0}^k {k+1 \choose j} B_j(1-t)^j A_{k+1-j}(t) \frac{x^{k+1}}{k!}\\
&= \sum_{k=0}^\infty \sum_{j=0}^k \frac{B_j (1-t)^j}{j!} \frac{A_{k+1-j}(t)}{(k+1-j)!} x^{k+1}\\
&= \sum_{m=0}^\infty B_m \frac{(x(1-t))^m}{m!} \sum_{n=0}^\infty A_{n+1}(t) \frac{x^{n+1}}{(n+1)!}\\
&=\frac{x(1-t)}{e^{x(1-t)}-1} \left( \frac{t-1}{t-e^{x(t-1)}}-1 \right)\\
&=\frac{t-1}{t-e^{x(t-1)}} \times xe^{x(t-1)}\\
&=\sum_{n=0}^\infty A_n(t) \frac{x^n}{n!} \sum_{m=0}^\infty (t-1)^m \frac{x^{m+1}}{m!}\\
&=\sum_{k=0}^\infty \sum_{j=0}^k {k \choose j} A_j(t)(t-1)^{k-j} \frac{x^{k+1}}{k!}.
\end{align*}}
Upon comparing the coefficients of $\frac{x^{k+1}}{k!}$, we get
\begin{align*}
S_k(t)&=\sum_{j=0}^k {k \choose j} A_j(t)(t-1)^{k-j}\\
&=\sum_{j=0}^{k-1} {k \choose j} A_j(t)(t-1)^{k-j} + A_k(t)\\
&=(t-1)A_k(t)+A_k(t)=tA_k(t),
\end{align*}
where in the penultimate step we used \eqref{recurrence relation of Eulerian polynomial}.  This completes the proof.  
\end{proof}
In the next section,  we present the proofs of our main results.

\section{Proof of main results}\label{proofs}
\begin{proof}[Theorem \ref{a new q-series identity}][]
We will start with the left hand side of Theorem \ref{a new q-series identity},  that is,  
\begin{align}\label{first step in proof of Theorem}
\sum_{n=1}^\infty \left(\sum_{m=1}^nm^k\right) q^n (q^{n+1})_\infty &= \sum_{n=1}^\infty \left(\sum_{m=1}^{n-1}m^k\right) q^n(q^{n+1})_\infty + \sum_{n=1}^\infty n^k q^n (q^{n+1})_\infty \nonumber\\
&= \sum_{n=1}^\infty \frac{1}{k+1}\sum_{j=0}^k {k+1 \choose j} B_j n^{k+1-j} q^n(q^{n+1})_\infty \nonumber\\
& + \sum_{n=1}^\infty \frac{(-1)^{n-1} q^{{n+1 \choose 2}}A_k(q^n)}{(1-q^n)^{k} (q)_n},
\end{align}
where in the final step we used \eqref{sum of k-th power} to obtain the first sum and  \eqref{Ramanujan-type for Uchimura} to get the second sum.  
Now we shall try to simplify the first sum.  Interchanging the sums and again making use of  the identity \eqref{Ramanujan-type for Uchimura},  one can see that 
\begin{align}\label{first term}
& \sum_{n=1}^\infty \frac{1}{k+1}\sum_{j=0}^k {k+1 \choose j} B_j n^{k+1-j} q^n(q^{n+1})_\infty  \nonumber \\
& = \frac{1}{k+1}\sum_{j=0}^k {k+1 \choose j} B_j \sum_{n=1}^\infty \frac{(-1)^{n-1} q^{{n+1 \choose 2}}A_{k+1-j}(q^n)}{(1-q^n)^{k+1-j} (q)_n} \nonumber\\
&=  \frac{1}{k+1}  \sum_{n=1}^\infty \frac{(-1)^{n-1} q^{{n+1 \choose 2}}}{(1-q^n)^{k+1} (q)_n}  \sum_{j=0}^k {k+1 \choose j} B_j(1-q^n)^j A_{k+1-j}(q^n)\nonumber\\
&=\sum_{n=1}^\infty \frac{(-1)^{n-1} q^{{n+1 \choose 2}}}{(1-q^n)^{k+1} (q)_n} S_k(q^n)=\sum_{n=1}^\infty \frac{(-1)^{n-1} q^{{n+1 \choose 2}} q^nA_k(q^n)}{(1-q^n)^{k+1} (q)_n},  
\end{align}
where in the last step we employed Lemma \ref{relation between B_n and A_n(t)}. Now utilizing \eqref{first term} in \eqref{first step in proof of Theorem},  we get
\begin{align*}
\sum_{n=1}^\infty \left(\sum_{m=1}^nm^k\right) q^n (q^{n+1})_\infty &= \sum_{n=1}^\infty \frac{(-1)^{n-1} q^{{n+1 \choose 2}} q^nA_k(q^n)}{(1-q^n)^{k+1} (q)_n} + \sum_{n=1}^\infty \frac{(-1)^{n-1} q^{{n+1 \choose 2}}A_k(q^n)}{(1-q^n)^{k} (q)_n}\\
&= \sum_{n=1}^\infty \frac{(-1)^{n-1} q^{{n+1 \choose 2}}A_k(q^n)}{(1-q^n)^{k} (q)_n} \left( \frac{q^n}{1-q^n}+1 \right)\\
&=\sum_{n=1}^\infty \frac{(-1)^{n-1} q^{{n+1 \choose 2}}A_k(q^n)}{(1-q^n)^{k+1} (q)_n}.
\end{align*}
This finishes  the proof of Theorem \ref{a new q-series identity}.
\end{proof}

\begin{proof}[Theorem \ref{new expression to Andrews Crippa Simon}][]
Given that $f(n)=\sum_{k\geq 0} c_k n^k$ is a polynomial in $n$.  Consider $F(x)$ to be the generating function for $f(n)$,  that is,  
\begin{align*}
F(x)&:=\sum_{n=1}^\infty f(n)x^n = \sum_{n=1}^\infty \sum_{k\geq 0} c_k n^k x^n
=\sum_{k\geq 0} c_k \sum_{n=1}^\infty n^k x^n = \sum_{k\geq 0} c_k \frac{xA_k(x)}{(1-x)^{k+1}},
\end{align*}
where in the last equality we used \eqref{eulerian polynomial}.  
Substituting $x=q^n$ in the above expression, we see that 
\begin{align}\label{function F(q^n)}
F(q^n) = \sum_{k\geq 0} c_k\frac{q^nA_k(q^n)}{(1-q^n)^{k+1}}.
\end{align}
Now define the generating function for the sequence $a_n(q)$ as follows: 
\begin{align*}
A(x)&:=\sum_{n=1}^\infty a_n(q) x^n\\
&=\sum_{n=1}^\infty \left(f(n)+(1-q^{n-1})a_{n-1}(q)\right) x^n\\
&=F(x) + xA(x) - xA(qx).
\end{align*}
Thus,  we obtain
\begin{align*}
A(x)=\frac{F(x)}{1-x} - \frac{x}{1-x}A(qx).
\end{align*}
This recurrence relation suggests that 
\begin{align*}
A(x) &= \sum_{n=0}^\infty \frac{(-1)^{n} F(q^nx) x^n q^{n \choose 2}}{(x)_{n+1}}.
\end{align*}
Put $x=q$ in the above expression and then use \eqref{function F(q^n)} to  see that 
\begin{align*}
A(q)&=\sum_{n=1}^\infty \frac{(-1)^{n-1} q^{n \choose 2}}{(q)_{n}} \sum_{k\geq 0} c_k\frac{q^nA_k(q^n)}{(1-q^n)^{k+1}}\\
&=\sum_{k\geq 0} c_k \sum_{n=1}^\infty \frac{(-1)^{n-1} A_k(q^n) q^{n+1 \choose 2}}{(1-q^n)^{k+1}(q)_{n}}.
\end{align*}
Now apply Theorem \ref{a new q-series identity} to get
\begin{align}
A(q)&=\sum_{k\geq 0} c_k \sum_{n=1}^\infty \left(\sum_{i=1}^ni^k\right) q^n(q^{n+1})_\infty \nonumber \\
&=\sum_{n=1}^\infty \left(\sum_{i=1}^n f(i)\right) q^n(q^{n+1})_\infty.   \label{final step1}
\end{align}
From the recurrence relation of the sequence of polynomials $a_i(q)$,  it is evident that 
\begin{align*}
& a_i(q)=f(i)+(1-q^{i-1})a_{i-1}(q)\\
\Longrightarrow~ & f(i) - a_i(q) + a_{i-1}(q)  = q^{i-1}a_{i-1}(q).  
\end{align*}
Now taking the sum over $i$ from $1$ to $n$,  then letting $n \rightarrow \infty$ and finally using \eqref{final step1},  the result follows.  
\end{proof}

\begin{proof}[Corollary \ref{a new term of limit to our generalization}][]
Observe that
\begin{align*}
\sum_{i=1}^n f_k(i) &= \sum_{i=1}^n \sum_{j=1}^k {k \choose j}(-1)^{j-1}i^{k-j},\\
&= \sum_{i=1}^n \left( i^k - (i-1)^k \right) = n^k.
\end{align*}
Now apply Theorem \ref{new expression to Andrews Crippa Simon} to complete the proof of the corollary.
\end{proof}

\begin{proof}[Corollary \ref{Limiting exp for Dilcher}][]
As we have taken $f_k(n) = {n-1 \choose k-1}$,  we find that
$$
\sum_{i=1}^n f_k(i)  = \sum_{i=1}^n {i-1 \choose k-1} = {n \choose k}. 
$$
Hence the proof of \eqref{limit form_Dilcher}  immediately follows from Theorem \ref{new expression to Andrews Crippa Simon}.  
\end{proof}

\begin{proof}[Theorem \ref{theorem for kappa_3, kappa_4 and kappa_5}][]
For simplicity,  throughout this proof,  let us denote $\gamma_n^*(1)$ as $X_n$.
We define  $e_{n,k}:=\mathbb{E}(X_n^k)$.  
Simon,  Crippa and Collenberg \cite{SCC1993} proved  that $e_{n,1}$ and $e_{n,2}$ satisfy the following recurrence relations,
\begin{align}
e_{n,1} &= 1 + (1- q^{n-1}) e_{n-1, 1},  \label{first}\\
e_{n, 2} &= 2 n e_{n, 1} - a_{n,  2}, \label{second}
\end{align}
where $$a_{n,2}= \sum_{i=1}^n f_2(i) \prod_{j=i}^{n-1} (1- q^j), ~~ \textrm{and}~~ f_2(i)= (2i-1).  $$
More generally,  for any fixed $k \geq 1$ and $n \geq 1$,   we will show that $e_{n,k}$ satisfies the 
 following relation,   
\begin{align}\label{general recurrence for e_{n,k}}
e_{n,k} 
&= \sum_{\ell=1}^{k} {k \choose \ell} (-1)^{\ell-1}n^{k-\ell} a_{n,  \ell},  
\end{align}
where
\begin{align}\label{defn of a_{n,k}}
a_{n,1}=e_{n,1},\hspace{1cm}a_{n,k}=\sum_{i=1}^nf_k(i)\prod_{j=i}^{n-1} (1-q^j),\hspace{1cm} \mathrm{for}\hspace{0.2cm} k\geq 2,
\end{align}
and
\begin{align}\label{defn of f_k}
f_k(i)=\sum_{j=1}^k {k \choose j} (-1)^{j-1} i^{k-j}=\sum_{j=0}^{k-1}{k\choose j}(-1)^{k-j-1}i^j.
\end{align}
It is easy to see that the relation \eqref{general recurrence for e_{n,k}} is true for $k=1$ as $e_{n,1}=a_{n,1}$ for all $n \geq 1$ by our assumption.   Moreover,   the relation \eqref{general recurrence for e_{n,k}} is also true for $k=2$ as we know that \eqref{second} holds for all $n \geq 1$.  Let us assume that the relation for $e_{m,  j}$ is true for any $1 \leq m < \infty$ when $ 3 \leq j \leq k-1$ and   $ 1\leq m \leq n-1$ for $ j=k$.  Now we shall show that the relation holds for $e_{n,k}$.  Since $e_{n,k}= \mathbb{E}(X_n^k)$,   we can write
{\allowdisplaybreaks
\begin{align}
e_{n,k}  & = \sum_{h=1}^n h^k P(X_n =h) \nonumber \\
& = \sum_{h=1}^n h^k q^{n-h} \prod_{i=1}^{h-1} (1- q^{n-i}) \nonumber \\
&= q^{n-1} + (1- q^{n-1}) \sum_{h=2}^n h^k q^{n-h} \prod_{i=2}^{h-1} (1- q^{n-i}) \nonumber \\
& =  q^{n-1} + (1- q^{n-1})   \sum_{h=1}^{n-1} (h+1)^k q^{n-h-1} \prod_{i=1}^{h-1} (1- q^{n-i-1}) \nonumber \\
& = q^{n-1} + (1- q^{n-1})   \sum_{h=1}^{n-1} \left( 1+ \sum_{j=1}^k  {k \choose j} h^j \right) q^{n-h-1} \prod_{i=1}^{h-1} (1- q^{n-i-1}) \nonumber \\
& = q^{n-1} + (1- q^{n-1}) \sum_{h=1}^{n-1} P(X_{n-1}=h) + (1- q^{n-1}) \sum_{h=1}^{n-1} \sum_{j=1}^k {k \choose j} h^j  P(X_{n-1}=h) \nonumber \\
& = 1 + (1- q^{n-1}) \sum_{j=1}^k {k \choose j} \sum_{h=1}^{n-1} h^j  P(X_{n-1}=h) \nonumber \\
& = 1 + (1- q^{n-1}) \sum_{j=1}^k {k \choose j} e_{n-1,  j}  \nonumber \\
& = 1 + (1- q^{n-1}) \sum_{j=1}^k {k \choose j}  \sum_{\ell= 1}^j {j \choose \ell} (-1)^{\ell -1} (n-1)^{j- \ell} a_{n-1,  \ell} \quad (\textrm{using inductive hypothesis}) \nonumber \\
& = 1 + (1- q^{n-1}) \sum_{\ell= 1}^k (-1)^{\ell -1} a_{n-1,  \ell} \sum_{j= \ell}^k {k \choose j} {j \choose \ell}  (n-1)^{j- \ell} \nonumber \\
&= 1 + (1- q^{n-1}) \sum_{\ell= 1}^k (-1)^{\ell -1} a_{n-1,  \ell} \sum_{j=0}^{k-\ell} {k \choose j+\ell} {j+\ell \choose \ell} (n-1)^j  \nonumber \\
& = 1 + (1- q^{n-1}) \sum_{\ell= 1}^k {k \choose \ell} (-1)^{\ell -1} a_{n-1,  \ell} \sum_{j=0}^{k-\ell} {k-\ell \choose j} (n-1)^j  \nonumber \\
& =  1 + (1- q^{n-1}) \sum_{\ell= 1}^k {k \choose \ell} (-1)^{\ell -1} a_{n-1,  \ell}~ n^{k - \ell}.  \label{final step}
\end{align}}
From \eqref{defn of a_{n,k}} and \eqref{defn of f_k},  one can easily observe that,  for any $k \in \mathbb{N}$,  the following recurrence relation for $a_{n,k}$ holds: 
\begin{align}\label{relation of a_{n,k}}
a_{n,k} = f_k(n) + (1- q^{n-1}) a_{n-1,  k},~~\textrm{where}~~\sum_{i=1}^n f_k(i) = n^k.   
\end{align}
 Use this recurrence relation in \eqref{final step} to see that 
\begin{align*}
e_{n, k} & = 1 +  \sum_{\ell= 1}^k {k \choose \ell} (-1)^{\ell -1} n^{k - \ell} \left( a_{n, \ell} - f_{\ell}(n)  \right) \\
& = 1 +  \sum_{\ell= 1}^k {k \choose \ell} (-1)^{\ell -1} n^{k - \ell}  a_{n, \ell} - \sum_{\ell= 1}^k {k \choose \ell} (-1)^{\ell -1} n^{k - \ell}  f_{\ell}(n). 
\end{align*}
It becomes clear at this juncture that  to prove \eqref{general recurrence for e_{n,k}} it is enough to show  
\begin{align*}
\sum_{\ell= 1}^k {k \choose \ell} (-1)^{\ell -1} n^{k - \ell}  f_{\ell}(n) =1.  
\end{align*}
Using the definition \eqref{defn of f_k} of $f_\ell(n)$,  we can see that 
\begin{align*}
\sum_{\ell= 1}^k {k \choose \ell} (-1)^{\ell -1} n^{k - \ell}  f_{\ell}(n) & = \sum_{\ell= 1}^k  \sum_{j=1}^\ell  {k \choose \ell}  {\ell \choose j}  (-1)^{\ell+j} n^{k - j}  \\ 
& = \sum_{j=1}^k \sum_{\ell=j}^k {k \choose j} {k-j \choose \ell-j} (-1)^{\ell+j} n^{k - j}  \\
&=  \sum_{j=1}^k {k \choose j}  n^{k-j}  \sum_{\ell=0}^{k-j} {k-j \choose \ell} (-1)^\ell \\
&= 1 + \sum_{j=1}^{k-1} {k \choose j}  n^{k-j} (1-1)^{k-j} \\
&=1.
\end{align*}
This completes the proof of the relation \eqref{general recurrence for e_{n,k}} for $e_{n,k}$.  
As we know that the sequence $a_{n,k}$ satisfies the relation \eqref{relation of a_{n,k}}, so by applying Corollary \ref{a new term of limit to our generalization},  we have 
\begin{align*}
\lim_{n\rightarrow \infty}\left(n^k- a_{n,k}(q)\right) = \sum_{n=1}^\infty n^k q^n(q^{n+1})_\infty.
\end{align*}
Further,  utilize Uchimura's identity \eqref{Uchimura's gen} 
to see that 
\begin{align}\label{Limit form of Bell polynomials}
\lim_{n\rightarrow \infty}\left(n^k- a_{n,k}(q)\right) =  Y_k(K_1,K_2,\dots,K_m),
\end{align}
where $Y_k$ denotes the Bell polynomial defined in \eqref{define Bell polynomial}.  
Now  we are ready to calculate the limiting value of the third cumulant,  that is,  
\begin{align*}
 \lim_{n\rightarrow \infty}\left(\kappa_3(X_n)\right) & = \lim_{n\rightarrow \infty}\mathbb{E}(X_n-\mathbb{E}(X_n))^3 \nonumber \\
&=\lim_{n\rightarrow \infty} \left( \mathbb{E}(X_n^3) -3\mathbb{E}(X_n)\mathbb{E}(X_n^2) + 2\mathbb{E}(X_n)^3 \right) \nonumber \\
&=\lim_{n\rightarrow \infty} \left( e_{n,3} -3e_{n,1}e_{n,2} + 2e_{n,1}^3 \right) \nonumber \\
&=-\lim_{n\rightarrow \infty} \left( (n^3- a_{n,3}) -3(n-a_{n,1})(n^2- a_{n,2}) + 2(n-a_{n,1})^3 \right) \nonumber \\
&=- \left(Y_3 - 3 Y_1 Y_2 + 2 Y_1^3 \right) \nonumber \\
&=-K_2,   
\end{align*}
where in the ante-penultimate step we have used the recurrence relation \eqref{general recurrence for e_{n,k}},  whereas   in the penultimate step we used \eqref{Limit form of Bell polynomials} and in the final step we used Dilcher's identity \cite[p.~ 85,  Equation (2.2)]{dilcher}.  This proves \eqref{kappa_3}. 

Now we shall calculate the limiting value of the fourth cumulant,  that is,  
\begin{align*}
 \lim_{n\rightarrow \infty}\left(\kappa_4(X_n)\right) & =\lim_{n\rightarrow \infty}\left( \mathbb{E} (X_n - \mathbb{E}(X_n))^4  -3\left(\mathbb{E}(X_n-\mathbb{E}(X_n))^2\right)^2 \right) \\
&=\lim_{n\rightarrow \infty} \left( \mathbb{E}(X_n^4) - 4 \mathbb{E}(X_n) \mathbb{E}(X_n^3) -3 \mathbb{E}(X_n^2)^2 + 12 \mathbb{E}(X_n)^2 \mathbb{E}(X_n^2) - 6 \mathbb{E}(X_n)^4  \right)  \\
&=\lim_{n\rightarrow \infty}\left( e_{n,4} -4e_{n,1}e_{n,3} -3e_{n,2}^2  + 12e_{n,2}e_{n,1}^2 - 6e_{n,1}^4 \right)  \\
&=\lim_{n\rightarrow \infty}  (n^4- a_{n,4}) -4(n-a_{n,1})(n^3- a_{n,3}) -3(n^2-a_{n,2})^2 +12(n^2- a_{n,2})(n-a_{n,1})^2  \\ 
&\quad -6(n-a_{n,1})^4 ) \\
& = (Y_4- 4 Y_1 Y_3 - 3 Y_2^2 + 12 Y_2 Y_1^2 - 6 Y_1^4 ) \\
&= K_4. 
\end{align*}
Here again, we made use of  \eqref{general recurrence for e_{n,k}},   \eqref{Limit form of Bell polynomials} and  
Dilcher's identity \cite[p.~ 85,  Equation (2.3)]{dilcher}.  This completes the proof of \eqref{kappa_4}.  
Finally,  we evaluate the limiting value of the fifth cumulant in the following way: 
{\allowdisplaybreaks
\begin{align*}
\lim_{n\rightarrow \infty}\left(\kappa_5(X_n)\right)&  = \lim_{n\rightarrow \infty}\left( \mathbb{E} (X_n - \mathbb{E}(X_n))^5 - 10 \mathbb{E}(X_n-\mathbb{E}(X_n))^3\mathbb{E}(X_n-\mathbb{E}(X_n))^2 \right) \\
&=\lim_{n\rightarrow \infty} ( \mathbb{E}(X_n^5) - 5 \mathbb{E}(X_n^4)\mathbb{E}(X_n) + 20 \mathbb{E}(X_n^3)\mathbb{E}(X_n)^2 - 40 \mathbb{E}(X_n^2)\mathbb{E}(X_n)^3   \\ 
&\hspace{2cm}+ 24\mathbb{E}(X_n)^5 - 10\mathbb{E}(X_n^3)\mathbb{E}(X_n^2) + 10\mathbb{E}(X_n^2)^2\mathbb{E}(X_n)) \nonumber\\
&=\lim_{n\rightarrow \infty} \left( e_{n,5} - 5e_{n,4}e_{n,1} +20e_{n,3}e_{n,1}^2 -40e_{n,2}e_{n,1}^3 +24e_{n,1}^5 -10e_{n,3}e_{n,2} +10e_{n,2}^2e_{n,1} \right)  \\
&=-\lim_{n\rightarrow \infty}( (n^5- a_{n,5}) -5(n^4- a_{n,4})(n- a_{n,1}) +20(n^3- a_{n,3})(n- a_{n,1})^2  \nonumber\\
&\hspace{1cm}-40(n^2- a_{n,2})(n-a_{n,1})^3 +24(n-a_{n,1})^5 -10(n^3-a_{n,3})(n^2-a_{n,2})  \\
&\hspace{2.5cm} +10(n^2-a_{n,2})^2(n-a_{n,1}) )  \\
&=- ( Y_5 - 5 Y_4 Y_1 + 20 Y_3 Y_1^2 - 40 Y_2 Y_1^3 + 24 Y_1^5 - 10 Y_3 Y_2 + 10 Y_2^2 Y_1 ) \\
&=- K_5.
\end{align*}}
Here also to obtain the final step we employed  Dilcher's identity \cite[p.~ 85,  Theorem 1]{dilcher} and in the penultimate step we invoked \eqref{Limit form of Bell polynomials}.  
\end{proof}

So far we have not been able to extend the above technique for higher cumulants.  However,   we use another method to prove the general case of the $t$-th cumulant for any $t$.

\section{Proof of Theorem \ref{Conjecture}: The general case}\label{general case}


We use the notations as in the proof of Theorem \ref{theorem for kappa_3, kappa_4 and kappa_5}.   Let  us define $Z_n:=n-X_n$,  where $X_n = \gamma_n^*(1)$. 
Let $\kappa_t$ denote the $t$-th cumulant and $e_{n,k} = \mathbb{E}(X_n^k)$.  For any random variable $X$ and constant $c$,  it is well-known that $\kappa_1(X+c)= \kappa_1(X)+c$,  and for $t \geq 2$ one has $\kappa_t(X+ c)= \kappa_t(X)$,  and $\kappa_t(c X)= c^t \kappa_t(X)$ for any $t\geq 1$.   Hence,  for $t \geq 2$,  one can see that 
$$
\lim_{n \to \infty} \kappa_t(Z_n)= \lim_{n\to\infty} (-1)^t\kappa_t(X_n).  
$$
Thus,  for $t\geq 2$,  Theorem \ref{Conjecture} is equivalent to the fact that
\begin{align}\label{Conjecture in term of Y_n}
\lim_{n \to \infty} \kappa_t(Z_n) = K_{t}(q),
\end{align}
where 
\begin{align*}
K_{t+1}(q)=\sum_{n=1}^\infty \sigma_{t}(n)q^n = \sum_{n=1}^\infty \frac{n^tq^n}{1-q^n}
\end{align*}
be the divisor generating function and for $t=1$, we already know
\begin{align*}
\lim_{n \to \infty} \kappa_1(Z_n)= \lim_{n\to\infty}(n- \kappa_1(X_n)) = K_{1}(q).
\end{align*}

We first prove a lemma that plays a vital role to prove \eqref{Conjecture in term of Y_n}.  
\begin{lemma}\label{lem1}
We have
\begin{align*}
\mathbb{E}(Z_n^k) = n^k - a_{n,k},
\end{align*}
where $a_{n,k}$ is defined as in \eqref{defn of a_{n,k}}.
\end{lemma}
\begin{proof}
As $Z_n= n-X_n$, so using Binomial expansion
\begin{align}\label{centeredmoment1}
\mathbb{E}(Z_n^k) &=\mathbb{E}\left((n-X_n)^k\right) = \sum_{j=0}^k \binom{k}{j} n^{k-j}(-1)^j\mathbb{E}( X_n^j) \nonumber \\
&= n^k + \sum_{j=1}^k \binom{k}{j} n^{k-j}(-1)^je_{n,j}.        
\end{align}
Now we use \eqref{general recurrence for e_{n,k}} in the right side of the above expression to get
\begin{align*}
\mathbb{E}(Z_n^k) &= n^k + \sum_{j=1}^k \binom{k}{j} n^{k-j}(-1)^j \sum_{\ell=1}^j \binom{j}{\ell}(-1)^{\ell-1} n^{j-\ell}a_{n,\ell} \nonumber \\
&=n^k+ \sum_{\ell=1}^k (-1)^{\ell-1} n^{k-\ell}a_{n,\ell}\sum_{j=\ell}^k \binom{k}{j}\binom{j}{\ell} (-1)^j \nonumber \\
&=n^k - \sum_{\ell=1}^k {k \choose \ell} n^{k-\ell}a_{n,\ell}\sum_{j=\ell}^k  {k-\ell \choose j-\ell} (-1)^{j-\ell}.
\end{align*}
By using binomial theorem, the inner sum gives value $1$ if $k=\ell$ and $0$ otherwise. Therefore, we get 
\begin{align*}
\mathbb{E}(n-X_n)^k = n^k-a_{n,k}.
\end{align*}
Now we are ready to prove \eqref{Conjecture in term of Y_n},  which is equivalent to Theorem \ref{Conjecture}.  To prove \eqref{Conjecture in term of Y_n}, we further derive an equivalent statement by multiplying both sides of \eqref{Conjecture in term of Y_n} by $\frac{z^t}{t!}$ and then summing over $t$ to obtain
\begin{align}\label{cgfy}
\lim_{n\to\infty} \sum_{t= 1}^\infty \kappa_t(Z_n) \frac{z^t}{t!} &=  \sum_{t= 1}^\infty K_{t}(q) \frac{z^t}{t!}.
\end{align}
Now we expand the sum on the right hand side to see
\begin{align}\label{cgf1}
\sum_{t= 1}^\infty \frac{z^t}{t!} K_{t}(q) &=  \sum_{t= 1}^\infty \frac{z^t}{t!} \sum_{k=1}^\infty \frac{k^{t-1}q^k}{1-q^k} \nonumber \\
&=\sum_{k=1}^\infty\frac{q^k}{1-q^k} \frac{1}{k}\sum_{t=1}^\infty \frac{(zk)^t}{t!} \nonumber \\
&=\sum_{k=1}^\infty\frac{q^k}{1-q^k} \frac{e^{kz}-1}{k} \nonumber \\
&=  \sum_{k=1}^\infty\frac{e^{kz}-1}{k} \sum_{\ell=1}^\infty  q^{k\ell} \nonumber \\
&=\sum_{\ell=1}^\infty \sum_{k=1}^\infty\frac{e^{kz}q^{k\ell}}{k}-\sum_{\ell=1}^\infty \sum_{k=1}^\infty\frac{q^{k\ell}}{k} \nonumber \\
&=-\sum_{\ell=1}^\infty \log(1-e^{z}q^\ell)+ \sum_{\ell=1}^\infty  \log(1-q^\ell) \nonumber \\
&=\log \prod_{\ell=1}^\infty \frac{(1-q^\ell)}{(1-e^{z}q^\ell)} = \log \frac{(q)_\infty}{(e^{z}q)_\infty}.
\end{align}
Combining \eqref{cgfy} and \eqref{cgf1},  we get 
\begin{align*}
\lim_{n\to\infty} \sum_{t= 1}^\infty \kappa_t(Z_n) \frac{z^t}{t!} = \log \frac{(q)_\infty}{(e^{z}q)_\infty}.
\end{align*}
Hence,  we only need to show the identity
\begin{align}\label{mgf2}
\lim_{n\to\infty}\exp\left(  \sum_{t= 1}^\infty \kappa_t(Z_n) \frac{z^t}{t!}\right) = \frac{(q)_\infty}{(e^{z}q)_\infty}.
\end{align}
From the definition of the cumulant generating function \eqref{cumulant generating function}, we know that
\begin{align*}
\lim_{n\to\infty} \mathbb{E}\left(e^{Z_n z}\right) = \lim_{n\to\infty}\exp\left(  \sum_{t= 1}^\infty \kappa_t(Z_n) \frac{z^t}{t!}\right).
\end{align*}
Thus, to prove \eqref{mgf2}, we have to show that 
\begin{align*}
\lim_{n\to\infty} \mathbb{E}\left(e^{Z_n z}\right) = \frac{(q)_\infty}{(e^{z}q)_\infty}.
\end{align*}
Now we start with
\begin{align}\label{mgf3}
\lim_{n\to\infty} \mathbb{E}\left(e^{Z_n z}\right)&= \lim_{n\to\infty} \mathbb{E}\left( e^{(n-X_n) z}\right) \nonumber \\
&=\lim_{n\to\infty} \sum_{k=0}^\infty\frac{z^k}{k!} \mathbb{E}\left((n-X_n)^k\right) \nonumber\\
&= 1 +  \sum_{k=1}^\infty \frac{z^k}{k!} \lim_{n\to\infty} (n^k-a_{n,k}) \nonumber \\
&= 1 +  \sum_{k=1}^\infty \frac{z^k}{k!} \sum_{\ell=1}^\infty \ell^kq^\ell(q^{\ell+1})_\infty,
\end{align}
where we used Lemma \ref{lem1} in the penultimate step and Corollary \ref{a new term of limit to our generalization} in the last step. 
Further, we use
\begin{align*}
\sum_{\ell=0}^\infty q^\ell(q^{\ell+1})_\infty =1
\end{align*}
in \eqref{mgf3} to obtain
{\allowdisplaybreaks
\begin{align*}
\lim_{n\to\infty} \mathbb{E}\left(e^{Z_n z}\right)&= \sum_{\ell=0}^\infty q^\ell(q^{\ell+1})_\infty +  \sum_{k=1}^\infty \frac{z^k}{k!} \sum_{\ell=1}^\infty \ell^kq^\ell(q^{\ell+1})_\infty \nonumber \\
 &=  (q)_\infty + \sum_{\ell=1}^\infty q^\ell(q^{\ell+1})_\infty +  \sum_{k=1}^\infty \frac{z^k}{k!} \sum_{\ell=1}^\infty \ell^kq^\ell(q^{\ell+1})_\infty \nonumber \\
&=  (q)_\infty +  \sum_{k=0}^\infty \frac{z^k}{k!} \sum_{\ell=1}^\infty \ell^kq^\ell(q^{\ell+1})_\infty \nonumber \\
&=  (q)_\infty + \sum_{\ell=1}^\infty q^\ell(q^{\ell+1})_\infty  \sum_{k=0}^\infty \frac{z^k}{k!} \ell^k \nonumber \\
&=(q)_\infty + \sum_{\ell=1}^\infty e^{\ell z} q^\ell(q^{\ell+1})_\infty \nonumber \\
&=\sum_{\ell=0}^\infty e^{\ell z} q^\ell(q^{\ell+1})_\infty \nonumber \\
&= \frac{(q)_\infty}{(e^z q)_\infty},
\end{align*}}
where in the last step, we used \cite[Equation 2.3]{uchimura87} with $x=e^z$. This completes the proof of Theorem \ref{Conjecture}.

\end{proof}

\section{concluding remarks}\label{remarks}

The work of Simon-Crippa-Collenberg gives an additional representation for Uchimura's identity \eqref{uchimura's identity},  namely,  
\begin{align}\label{4-term of Uchimura}
\lim_{n\rightarrow \infty}(n-\mathbb{E}(\gamma^*_n(1)) = \sum_{n=1}^\infty nq^n(q^{n+1})_\infty=\sum_{n=1}^\infty \frac{(-1)^{n-1} q^{ \frac{n(n+1)}{2}}}{(1-q^n)(q)_n}=\sum_{n=1}^\infty d(n) q^n.  
\end{align}
Further,  Andrews,  Crippa and Simon showed that 
\begin{align*}
\lim_{n\rightarrow \infty} \mathrm{Var}(\gamma^*_n(1))=\sum_{n=1}^\infty \sigma(n)q^n= K_2(q).  
\end{align*}
Uchimura generalized his own identity in the following way: 
\begin{align*}
\sum_{n=1}^\infty n^k q^n(q^{n+1})_\infty = Y_k(K_1,K_2,\dots,K_k).  
\end{align*}
Now combining the identity \eqref{Ramanujan-type for Uchimura} and Corollary \ref{a new term of limit to our generalization},  we see that 
\begin{align}\label{limit of Uchimura}
\lim_{n\rightarrow \infty}\left(n^k- a_{n,k}(q)\right) = \sum_{n=1}^\infty n^k q^n(q^{n+1})_\infty = \sum_{n=1}^\infty \frac{(-1)^{n-1} q^{ {n+1 \choose 2 } } A_k(q^n)}{(1-q^n)^k (q)_n}= Y_k(K_1,K_2,\dots,K_k),
\end{align}
where $a_{n,k}$ is defined as in \eqref{definition a_n,k}.  Furthermore,  we obtained a limiting expression for Dilcher's identity,   namely,  for $k \in \mathbb{N}$,  
\begin{align}\label{4 term for Dilcher}
\lim_{n \rightarrow \infty} \left(  {n \choose k} - a_{n, k}(q)  \right) = \sum_{n=k}^\infty {n\choose k} q^n(q^{n+1})_\infty & =\sum_{n=1}^\infty \frac{(-1)^{n-1}q^{{n+k \choose 2}-{k\choose 2}}}{(1-q^n)^k (q)_n} \nonumber \\ 
& =\sum_{j_1=1}^\infty\frac{q^{j_1}}{1-q^{j_1}} \cdots \sum_{j_k=1}^{j_{k-1}}\frac{q^{j_k}}{1-q^{j_k}},
\end{align}
where $a_{n,k}$ is defined in \eqref{a_n,k for Dilcher}.
As observed in \eqref{4-term of Uchimura}-\eqref{4 term for Dilcher},  we discovered a limit expression for Uchimura's identity \eqref{Uchimura's gen} and Dilcher's  identity \eqref{dilcher 1},  which are generalizations of \eqref{uchimura's identity}.  Thus,  it would be fascinating to find limit expressions for all the generalizations of Uchimura's identity \eqref{uchimura's identity} in the framework of Uchimura-Ramanujan-divisor type identities studied in \cite{ABEM24}.   

Furthermore,  we also generalized the work of Simon-Crippa-Collenberg and  Andrews-Crippa-Simon by showing that the limit of the $t$-th cumulant is nothing but the generalized divisor function $K_t(q)$,  i.e.,  for any $t \geq 1$,  
\begin{align*}
\lim_{n\rightarrow \infty}\kappa_t(Z_n)&= K_t(q),
\end{align*}
where $Z_n= n-  \gamma_n^*(1)$.   In the proof of Theorem \ref{theorem for kappa_3, kappa_4 and kappa_5},  we evaluated the limit of the third,  fourth and fifth cumulants using the limit form \eqref{Limit form of Bell polynomials}  of Uchimura's identity and Dilcher's identity \cite[p.~85,  Equations (2.1)--(2.3)]{dilcher}.  It would be interesting to utilize this approach to give another proof of the general case of the $t$-th cumulant.


\section*{Acknowledgements}  The first author wishes to thank University Grant Commission (UGC), India, for providing Ph.D. scholarship.  The second author wants to thank the Department of Mathematics, Pt. Chiranji Lal Sharma Government College, Karnal, for providing research facilities. The third author is thankful to his institution BITS Pilani for providing the New Faculty Seed Grant NFSG/PIL/2024/P3797. The fourth author is grateful to the Anusandhan National Research Foundation (ANRF),  India, for giving the Core Research Grant CRG/2023/002122 and MATRICS Grant MTR/2022/000545.   
\\

\section*{Conflict of Interest}
The authors declare that they do not have any conflict of interest.

\end{document}